\documentclass[12pt, a4paper]{article}
\usepackage{amssymb, amsmath,amsthm }
\usepackage[all]{xy}
\newtheorem{prop}{Proposition}[section]

\newtheorem{lem}[prop]{Lemma}
\newtheorem{thm}[prop]{Theorem}

\newtheorem{cor}[prop]{Corollary}

\DeclareMathAlphabet{\mathpzc}{OT1}{pzc}{m}{it}

\DeclareMathOperator{\Hom}{Hom}
\DeclareMathOperator{\Ind}{Ind}
\DeclareMathOperator{\cInd}{c-Ind}

\DeclareMathOperator{\GL}{GL}

\DeclareMathOperator{\Gal}{Gal}

\DeclareMathOperator{\Sym}{Sym}
\DeclareMathOperator{\Rep}{Rep}
\DeclareMathOperator{\val}{val}

\DeclareMathOperator{\Fil}{Fil}
\DeclareMathOperator{\Supp}{Supp}
\newcommand{\cIndu}[3]{\cInd_{#1}^{#2}{#3}}

\newcommand{\Indu}[3]{\Ind_{#1}^{#2}{#3}}
\newcommand{\KK}{\mathfrak{K}}

\newcommand{\FF}{\mathcal F}

\newcommand{\pF}{\mathfrak{p}_{F}}

\newcommand{\ZZ}{\mathbb{Z}}

\newcommand{\DD}{\mathcal{D}}

\newcommand{\Qpbar}{\overline{\mathbb Q}_p}
\newcommand{\Qp}{\mathbb {Q}_p}

\newcommand{\Zp}{\mathbb{Z}_p}

\newcommand{\oL}{\mathfrak o_L}
\newcommand{\pL}{\mathfrak p_L}
\newcommand{\GG}{\mathcal G}
\author{Vytautas Pa\v{s}k\={u}nas}
\title{On some crystalline representations of $\GL_2(\Qp)$}

\begin{document} 

\maketitle
\begin{abstract}
 We show that  the universal unitary completion of certain locally algebraic representation of  
$G:=\GL_2(\Qp)$ with $p>2$
is non-zero, topologically irreducible, admissible and corresponds to a 
$2$-dimensional crystalline representation with non-semisimple Frobenius  via the $p$-adic Langlands correspondence for $G$.
\end{abstract}  
\section{Introduction}

Let $G:=\GL_2(\Qp)$ and $B$ be the subgroup of upper-triangular matrices in $G$. Let $L$ be a finite extension of $\Qp$. 

\begin{thm}\label{A} Assume that $p>2$, let $k\ge 2$ be an integer  and let $\chi:\Qp^{\times}\rightarrow L^{\times}$ a smooth character with 
$\chi(p)^2 p^{k-1}\in \oL^{\times}$. Assume that there exists a $G$-invariant norm $\|\centerdot\|$ on 
$(\Indu{B}{G}{\chi\otimes \chi|\centerdot|^{-1}})\otimes \Sym^{k-2} L^2$.
Then the completion 
$E$ is a topologically irreducible, admissible Banach space representation of $G$. Moreover, if we let $E^0$ be the unit ball in $E$ then 
$$ V_{k,2\chi(p)^{-1}}\otimes (\chi|\chi|)\cong L\otimes_{\oL} \underset{\longleftarrow} {\lim}\ \mathbf V( E^0/\varpi^n_L E^0),$$
where $\mathbf V$ is Colmez's Montreal functor, and $V_{k, 2\chi(p)^{-1}}$, is a $2$-dimensional irreducible crystalline representation of 
$\mathcal G_{\Qp}$
the absolute Galois group of $\Qp$, with Hodge-Tate weights $(0, k-1)$ and the trace of crystalline Frobenius equal to $2\chi(p)^{-1}$. 
\end{thm}
 
As we explain in \S\ref{exist}, the existence of such $G$-invariant norm follows from the recent work of Colmez, \cite{colmez1}.
Our  result  addresses Remarque 5.3.5 in \cite{bergerbreuil}. In other words, 
the  completion $E$  fits into the $p$-adic Langlands correspondence for $\GL_2(\Qp)$. 

The idea is to ``approximate''  $(\Indu{B}{G}{\chi\otimes \chi|\centerdot|^{-1}})\otimes \Sym^{k-2} L^2$ with representations 
 $(\Indu{B}{G}{\chi\delta_x\otimes \chi\delta_{x^{-1}}|\centerdot|^{-1}})\otimes \Sym^{k-2} L^2$, where 
$\delta_x:\Qp^{\times}\rightarrow L^{\times}$ 
is an unramified character with $\delta_x(p)=x\in 1+\pL$. If $x^2\neq 1$ then 
$\chi\delta_x \neq \chi\delta_{x^{-1}}$ and the analog of Theorem \ref{A} is a result of 
Berger-Breuil \cite{bergerbreuil}.  This allows to  deduce admissibility.
This ``approximation'' process relies on the results of Vign\'eras \cite{vig}.  Using Colmez's functor 
$\mathbf V$ we may then transfer the question of irreducibility 
to the Galois side. Here, we use the fact that for  $p>2$ the representation $V_{k, \pm 2 p^{(k-1)/2}}$ sits in the $p$-adic family studied by 
Berger-Li-Zhu in
\cite{bergerlizhu}.  

\textit{Acknowledgements.} I thank Laurent Berger, Christophe Breuil, Ga\"etan Chenevier and Pierre Colmez for answering my questions. I also thank 
Guy Henniart, Ariane M\'ezard and Rachel Ollivier for organizing 'Groupe de Travail  sur les repr\'esentations $p$-adiques de $\GL_2(\Qp)$', where
I learnt about Colmez's functor.  This paper was written when I was visiting IH\'ES and 
Universit\'e Paris-Sud, supported by Deutsche Forschungsgemeinschaft. I would like to thank these institutions. 

\section{Notation}
We fix an algebraic closure $\Qpbar$ of $\Qp$. We let $\val$ be the valuation on $\Qpbar$ such that $\val(p)=1$, and we set $|x|:=p^{-\val(x)}$. Let 
$L$ be a finite extension of $\Qp$ contained in $\Qpbar$, $\oL$ the ring of integers of $L$, $\varpi_L$ a uniformizer, and $\pL$ the maximal ideal of $\oL$.
Given a character $\chi:\Qp^{\times}\rightarrow L^{\times}$ we consider $\chi$ as a character of the absolute Galois group $\GG_{\Qp}$ of $\Qp$ via the 
local class field theory by sending the geometric Frobenius to $p$.  

Let $G:=\GL_2(\Qp)$, $B$ the subgroup of upper-triangular matrices. Given two characters $\chi_1, \chi_2:\Qp^{\times}\rightarrow L^{\times}$ we consider 
$\chi_1\otimes\chi_2$ as a character of $B$, which sends a matrix $\bigl ( \begin{smallmatrix} a & b \\ 0 & d\end{smallmatrix} \bigr )$ to 
$\chi_1(a) \chi_2(d)$. Let $Z$ be the centre of $G$, $K:=\GL_2(\Zp)$, $I:=\bigl (\begin{smallmatrix} \Zp^{\times} & \Zp \\ p\Zp & \Zp^{\times}\end{smallmatrix}\bigr )$
and for $m\ge 1$ we define  
$$K_m:=\begin{pmatrix} 1 +p^m \Zp & p^m \Zp \\ p^m \Zp & 1+p^m \Zp\end{pmatrix}, \quad 
I_m:=\begin{pmatrix}   1 +p^m \Zp & p^{m-1} \Zp \\ p^m \Zp & 1+p^m \Zp\end{pmatrix}.$$ 
Let $\KK_0$ be the $G$-normalizer of $K$, so that $\KK_0=KZ$, and $\KK_1$ the $G$-normalizer of $I$, so that $\KK_1$ is generated as a group 
by $I$ and $\Pi:=\bigl ( \begin{smallmatrix} 0 & 1\\ p & 0\end{smallmatrix}\bigr )$. We note that if $m\ge 1$ then $K_m$ is normal in $\KK_0$ and 
$I_m$ is normal in $\KK_1$. We denote $s:= \bigl ( \begin{smallmatrix} 0 & 1\\ 1 & 0\end{smallmatrix}\bigr )$.

\section{Diagrams}
Let $R$ be a commutative ring, (typically $R=L$, $\oL$ or $\oL/\pL^n$). By a diagram $D$ of $R$-modules, we mean the data $(D_0,D_1,r)$, where 
$D_0$ is a $R[\KK_0]$-module, $D_1$ is $R[\KK_1]$-module and $r: D_1\rightarrow D_0$ is a $\KK_0\cap \KK_1=IZ$-equivariant homomorphism of $R$-modules. 
A morphism $\alpha$ between two diagrams $D$, $D'$ is given by $(\alpha_0,\alpha_1)$, where $\alpha_0: D_0\rightarrow D'_0$ is a morphism of 
$R[\KK_0]$-modules, $\alpha_1: D_1\rightarrow D'_1$ is a morphism of $R[\KK_1]$-modules, and the diagram 
\begin{equation}\label{commute}
\xymatrix{ D_0\ar[r]^{\alpha_0}& D'_0\\ D_1\ar[u]^r\ar[r]^{\alpha_1} & D'_1\ar[u]_{r'} }
\end{equation}
commutes in the category of $R[IZ]$-modules. The condition \eqref{commute} is  important, since one can have two diagrams of $R$-modules $D$ and
$D'$, such that $D_0\cong D'_0$ as $R[\KK_0]$-modules, $D_1\cong D_1'$ as $R[\KK_1]$-modules, however $D\not\cong D'$ as diagrams.    
The diagrams of $R$-modules with the above morphisms  form an abelian category. To a diagram $D$ one may associate a complex of $G$-representations:
\begin{equation}\label{complex}
\cIndu{\KK_1}{G}{D_1\otimes\delta}\overset{\partial}{\longrightarrow} \cIndu{\KK_0}{G}{D_0},
\end{equation}  
where $\delta: \KK_1\rightarrow R^{\times}$ is the character $\delta(g):=(-1)^{\val(\det g)}$; $\cIndu{\KK_i}{G}{D_i}$ denotes the space of functions
$f:G\rightarrow D_i$, such that $f(kg)= k f(g)$, for $k\in \KK_i$ and $g\in G$, and $f$ is supported only on finitely many cosets $\KK_i g$. To 
describe $\partial$, we note that Frobenius reciprocity gives $\Hom_G( \cIndu{\KK_1}{G}{D_1\otimes\delta},\cIndu{\KK_0}{G}{D_0})\cong 
\Hom_{\KK_1}(D_1\otimes \delta, \cIndu{\KK_0}{G}{D_0})$, now $\Indu{IZ}{\KK_1}{D_0}$ is a direct summand of the restriction of 
$\cIndu{\KK_0}{G}{D_0}$ to $\KK_1$, and $\Hom_{\KK_1}(D_1\otimes\delta, \Indu{IZ}{\KK_1}{D_0})\cong \Hom_{IZ}(D_1, D_0)$, since $\delta$ is trivial 
on $IZ$. Composition of  the above maps yields  a map $\Hom_{IZ}(D_1, D_0)\rightarrow \Hom_G( \cIndu{\KK_1}{G}{D_1\otimes\delta},\cIndu{\KK_0}{G}{D_0})$, 
we let $\partial$ be the image of $r$. We define $H_0(D)$ to be the cokernel of $\partial$ and $H_1(D)$ to be the kernel of $\partial$. So we have an
exact sequence of $G$-representations:
\begin{equation}\label{homology}
0\rightarrow H_1(D)\rightarrow \cIndu{\KK_1}{G}{D_1\otimes\delta}\overset{\partial}{\rightarrow} \cIndu{\KK_0}{G}{D_0}\rightarrow H_0(D)\rightarrow 0
\end{equation}
Further, if $r$ is injective then one may show that $H_1(D)=0$, see \cite[Prop. 0.1]{vig}. To a diagram $D$ one may associate a $G$-equivariant coefficient
system $\mathcal V$ of $R$-modules on the Bruhat-Tits tree, see \cite[\S5]{coeff}, then $H_0(D)$ and $H_1(D)$ compute the homology of the coefficient system 
$\mathcal V$ and the map $\partial$ has a natural interpretation. Assume that $R=L$ (or any field of characteristic $0$), 
and let $\pi$ be a smooth irreducible representation of $G$ on an $L$-vector space, so that for all $v\in \pi$ the subgroup 
$\{g\in G: gv=v\}$ is open in $G$. Since the action of $G$ is smooth there exists an $m\ge 0$ such that $\pi^{I_m}\neq 0$. To $\pi$ we may 
associate a diagram $D:=(\pi^{I_m}\hookrightarrow \pi^{K_m})$. As a very special case of a result by Schneider and Stuhler \cite[Thm V.1]{ss}, 
\cite[\S3]{ss1}, we obtain that $H_0(D)\cong \pi$.    
  
We are going to compute such diagrams $D$, attached to smooth principal series representations of $G$ on $L$-vector spaces. Given  smooth 
characters $\theta_1, \theta_2:\Zp^{\times}\rightarrow L^{\times}$ and $\lambda_1, \lambda_2\in L^{\times}$ we define a diagram 
$D(\lambda_1, \lambda_2, \theta_1, \theta_2)$ as follows. Let $c\ge 1$ be an integer, such that $\theta_1$ and $\theta_2$ are trivial 
on $1+ p^c \Zp$. We set $J_c:= (K\cap B)K_c= (I\cap B)K_c$, so that $J_c$ is a subgroup of $I$. We let $\theta: J_c \rightarrow L^{\times}$ be the 
character:
$$\theta(\begin{pmatrix} a & b \\ c & d \end{pmatrix}):=\theta_1(a) \theta_2(d).$$  
We let $D_0:=\Indu{J_c}{K}{\theta}$, and we let $p\in Z$ act on $D_0$ by a scalar $\lambda_1\lambda_2$, so that $D_0$ is a representation of $\KK_0$.
We set $D_1:=D_0^{I_c}$ so that $D_1$ is naturally a representation of $IZ$. We are going to put an action of $\Pi$ on $D_1$, so that $D_1$ is a 
representation of $\KK_1$. Let 
\begin{equation}\label{V1s}
V_1:=\{f\in D_1: \Supp f \subseteq I\}, \quad V_s:= \{f\in D_1: \Supp f \subseteq J_c s I\}.
\end{equation}
 Since $I$ contains $K_1$ 
we have $J_c s I=  (B\cap K) s I= I sI$, hence $D_1=V_1\oplus V_s$. For all $f_1\in V_1$  and $f_s\in V_s$, 
we define $\Pi \centerdot f_1\in V_s$  and $\Pi \centerdot f_s\in V_1$
such that  
\begin{equation}\label{actionPi}
[\Pi \centerdot f_1](s g):= \lambda_1 f_1(\Pi^{-1} g \Pi), \quad [\Pi \centerdot f_s](g)=\lambda_2 f_s(s\Pi g \Pi^{-1}), \quad \forall g\in I;
\end{equation} 
Every $f\in D_1$ can be written uniquely as $f=f_1+ f_s$, with $f_1\in V_1$ and $f_s\in V_s$, and we define $\Pi\centerdot  f:=\Pi\centerdot f_1+
\Pi\centerdot f_s$.

\begin{lem}\label{princdiag} The equation \eqref{actionPi} defines an action of $\KK_1$ on $D_1$. We denote the diagram $D_1\hookrightarrow D_0$
by $D(\lambda_1, \lambda_2, \theta_1, \theta_2)$. Moreover, let $\pi:=\Indu{B}{G}{\chi_1\otimes \chi_2}$ be a smooth principal series representation 
of $G$, with 
$\chi_1(p)=\lambda_1$, $\chi_2(p)=\lambda_2$, $\chi_1|_{\Zp^{\times}}=\theta_1$ and $\chi_2|_{\Zp^{\times}}=\theta_2$.     
There exists an isomorphism of diagrams $D(\lambda_1,\lambda_2, \theta_1, \theta_2)\cong (\pi^{I_c}\hookrightarrow \pi^{K_c})$. In particular, we have a 
$G$-equivariant isomorphism $H_0(D(\lambda_1, \lambda_2, \theta_1, \theta_2))\cong \pi$. 
\end{lem}
\begin{proof} We note that $p\in Z$ acts on $\pi$ by a scalar $\lambda_1 \lambda_2$. Since $G=BK$, 
we have $\pi|_K \cong \Indu{B\cap K}{K}{\theta}$, and so  the map $f\mapsto [g\mapsto f(g)]$ induces an isomorphism 
$\iota_0:\pi^{K_c}\cong \Indu{J_c}{K}{\theta}=D_0$. Let $\FF_1:=\{f\in \pi: \Supp f \subseteq BI\}$ and $\FF_s:=\{f\in \pi: \Supp f\subseteq BsI\}$. 
Iwasawa decomposition gives $G=BI \cup BsI$, hence $\pi=\FF_1\oplus \FF_s$.
If $f_1\in \FF_1$ then $\Supp (\Pi f_1)= (\Supp f_1)\Pi^{-1}\subseteq B I \Pi^{-1}= Bs I$. Moreover, 
\begin{equation}\label{act1}
[\Pi f_1](sg)= f_1(sg \Pi)= f_1(s\Pi (\Pi^{-1} g \Pi))= \chi_1(p) f_1(\Pi^{-1} g \Pi), \quad \forall g\in I   
\end{equation}
Similarly, if $f_s\in \FF_s$ then $\Supp (\Pi f_s)=(\Supp f_s)\Pi^{-1}\subseteq  Bs I\Pi^{-1}= BI$, and 
\begin{equation}\label{acts}
[\Pi f_s](g)= f_1(g \Pi)= f_1((\Pi s)s (\Pi^{-1} g \Pi))= \chi_2(p) f_s(s (\Pi^{-1} g \Pi)), \quad \forall g\in I   
\end{equation}  
Now $\pi^{I_c}=\FF_1^{I_c} \oplus \FF_s^{I_c}\subset \pi^{K_c}$. Let $\iota_1$ be the restriction of $\iota_0$ to $\pi^{I_c}$ then  it is 
immediate that $\iota_1(\FF_1^{I_c})=V_1$ and $\iota_1(\FF_s^{I_1})=V_s$, where $V_1$ and $V_s$ are as above. Moreover, if $f\in D_1$ and 
$\Pi\centerdot f$ is given by \eqref{actionPi} then $\Pi \centerdot f= \iota_1( \Pi \iota_1^{-1}(f))$. 
Since $\KK_1$ acts on $\pi^{I_c}$, we get that \eqref{actionPi}
defines an action of $\KK_1$ on $D_1$, such that $\iota_1$ is $\KK_1$-equivariant. Hence, $(\iota_0, \iota_1)$ is an isomorphism of 
diagrams $(\pi^{I_c}\hookrightarrow \pi^{K_c})\cong (D_1\hookrightarrow D_0)$.
\end{proof}   

\section{Main result}
In this section we prove the main result.
\begin{lem}\label{lattice} Let $U$ be a finite dimensional $L$-vector space with subspaces $U_1, U_2$ such that $U=U_1\oplus U_2$. 
For $x\in L$ define a map
$\phi_x: U\rightarrow U$, $\phi_x(v_1+v_2)= xv_1+ v_2$, for all $v_1\in U_1$ and  $v_2\in U_2$. Let $M$ be an $\oL$-lattice in $V$, then 
there exists an integer $a\ge 1$ such that for $x\in 1+\pL^a$ we have $\phi_x(M)=M$.
\end{lem}
\begin{proof} Let $N$ denote the image of $M$ in $U/U_2$. Then $N$ contains $(M\cap U_1) + U_2$, and both are lattices in $U/U_2$.
Let $a\ge 1$ be the smallest integer, such that 
$\pL^{-a} (M\cap U_1) + U_2$ contains $N$. Suppose that $x\in 1+\pF^a$ and $v\in M$. We may write $v=\lambda v_1 +v_2$, with $v_1\in M\cap U_1$, 
$v_2\in U_2$ 
and $\lambda\in \pL^{-a}$. Now $\phi_x(v)= v + \lambda (x-1) v_1\in M$. Hence we get $\phi_x(M)\subseteq M$ and $\phi_{x^{-1}}(M)\subseteq M$. Applying 
$\phi_{x^{-1}}$ to the first inclusion gives $M\subseteq \phi_{x^{-1}}(M)$.
\end{proof} 

We fix an integer $k\ge 2$ and set $W:=\Sym^{k-2} L^2$, an algebraic representation of $G$. Let 
$\pi:=\pi(\chi_1, \chi_2):=\Indu{B}{G}{\chi_1\otimes \chi_2}$ be a
smooth principal series $L$-representation of $G$. We say that $\pi\otimes W$ admits a $G$-invariant norm, if there exists a norm 
$\|\centerdot\|$ on $\pi\otimes W$, with respect to which $\pi\otimes W$ is a normed $L$-vector space, such that  $\|g v\|=\|v\|$, 
for all $v\in \pi\otimes W$ and $g\in G$. 

Let $c\ge 1$ be an integer such that both $\chi_1$ and $\chi_2$ are trivial on $1+p^c \Zp$. Let $D$ be the diagram 
$\pi^{I_c}\otimes W\hookrightarrow \pi^{K_c}\otimes W$. Since  
$H_0(\pi^{I_c}\hookrightarrow \pi^{K_c})\cong \pi$, by tensoring \eqref{complex} with $W$ we obtain $H_0(D)\cong \pi\otimes W$. 
Assume that $\pi\otimes W$ admits a $G$-invariant  norm $\|\centerdot\|$, set $(\pi\otimes W)^0:=\{v\in \pi\otimes W: \|v\|\le 1\}$. 
Then we may define a diagram $\mathcal D=(\mathcal D_1\hookrightarrow \mathcal D_0)$ of $\oL$-modules:
$$\mathcal D:=( (\pi^{I_c}\otimes W)\cap (\pi\otimes W)^0\hookrightarrow (\pi^{K_c}\otimes W)\cap (\pi\otimes W)^0 ).$$ 
In this case Vign\'eras \cite{vig} has shown that  the inclusion $\mathcal D \hookrightarrow D$ induces a $G$-equivariant  
injection 
$H_0(\mathcal D)\hookrightarrow H_0(D)$, such that $H_0(\mathcal D)\otimes_{\oL} L= H_0(D)$; $H_1(\mathcal D)=0$. 
Moreover, $H_0(\mathcal D)$ does not contain an $\oL$-submodule isomorphic to $L$, see \cite[Prop 0.1]{vig}. 
Since $H_0(D)$ is an $L$-vector space of countable dimension, this implies that $H_0(\mathcal D)$ is a free $\oL$-module. By tensoring 
\eqref{complex} with $\oL/\pL^n$ we obtain 
\begin{equation}\label{Hmod}
H_0(\mathcal D)\otimes_{\oL} \oL/\pL^n\cong H_0(\mathcal D\otimes_{\oL} \oL/\pL^n).
\end{equation}

\begin{prop}\label{adjust} Let $\pi=\pi(\chi_1,\chi_2)$ be a smooth principal series representation, assume that 
$\pi\otimes W$ admits a $G$-invariant norm and let $\mathcal D$ be as above. 
Then there exists an integer $a\ge 1$ such that for all $x\in 1+\pF^b$, with $b\ge a$, there
exists a finitely generated $\oL[G]$-module $M$ in $\pi(\chi_1\delta_{x^{-1}}, \chi_2\delta_{x})\otimes W$, which 
is free as an $\oL$-module and a $G$-equivariant isomorphism 
$$M\otimes_{\oL}\oL/\pL^b \cong H_0(\mathcal D)\otimes_{\oL}  \oL/\pL^b,$$
where $\delta_x: \Qp^{\times}\rightarrow L^{\times}$ is an unramified character with $\delta_x(p)=x$.
\end{prop}
\begin{proof} Apply Lemma \ref{lattice} to $U=D_1$, $U_1=V_1\otimes W$, $U_2= V_{s}\otimes W$ and $M=\mathcal D_1$, where $V_1$ and $V_s$ are given by 
\eqref{V1s}. Then we get an integer $a\ge 1$, such that
for all $x\in 1+\pL^a$, $\phi_{x}(\mathcal D_1)=\mathcal D_1$. It is immediate that $\phi_x$ is $IZ$-equivariant. We define a new action $\star$ of 
$\Pi$ on $D_1$, by 
setting $\Pi\star v: = \phi_x (\Pi \phi_x^{-1}(v))$. This gives us a new diagram $D(x)$, so that $D(x)_0=D_0$ as a representation of 
$\KK_0$, $D(x)_1=D_1$ as a representation
 of $IZ$, the $IZ$-equivariant injection $D(x)_1\hookrightarrow D(x)_0$ is equal to the $IZ$-equivariant injection $D_1\hookrightarrow D_0$, 
but the action of $\Pi$ 
on $D_1$ is given by $\star$, (here by $=$ we really mean an equality, not an isomorphism). 
If $f_1\in V_1$ and $f_s\in V_s$ then 
$$\Pi\star (f_1\otimes w)= f'_s\otimes (\Pi w), \quad \Pi\star (f_s\otimes w)= f_1'\otimes (\Pi w), \quad \forall w\in W,$$
where $f'_s\in V_s$, $f'_1\in V_1$ and for all $g\in I$ we have:
\begin{equation}\label{checkstar1} 
f_s'(s g)= x^{-1} [\Pi\centerdot f_1](sg)= x^{-1} \lambda_1 f_1(\Pi^{-1} g \Pi),
\end{equation}
\begin{equation}\label{checkstars}
 f_1'(g)=x [\Pi \centerdot f_s](g)=x\lambda_2 f_s(s\Pi g \Pi^{-1}).
\end{equation}
Hence, we have an isomorphism of diagrams $D(x)\cong D(x^{-1}\lambda_1, x\lambda_2, \theta_1, \theta_2)$ and so Lemma \ref{princdiag} gives
$H_0(D(x))\cong \pi(\chi_1\delta_{x^{-1}}, \chi_2\delta_x)\otimes W$. Now, let $b\ge a$ be an integer and suppose that $x\in 1+\pL^b$.
Since, $\Pi \centerdot \mathcal D_1=\phi_x(\mathcal D_1)=\phi_x^{-1}(\mathcal D_1)=\mathcal D_1$ we get 
$$\Pi\star (\mathcal D_0 \cap D_1)= \Pi\star \mathcal D_1= \phi_x(\Pi \phi_x^{-1}(\mathcal D_1))= \mathcal D_1.$$      
So if we let $\mathcal D(x)_0:=\mathcal D_0$ and $\mathcal D(x)_1:= \mathcal D(x)_0\cap D(x)_1$, where $\Pi$ acts on $\mathcal D(x)_1$ by $\star$ 
then the diagram 
$\mathcal D(x):= (\mathcal D(x)_1\hookrightarrow \mathcal D(x)_0)$ is an integral structure in $D(x)$ in the sense of \cite{vig}. 
The results of Vign\'eras cited above imply 
that $M:=H_0(\mathcal D(x))$ is a finitely generated $\oL[G]$-submodule of $\pi(\chi_1\delta_{x^{-1}}, \chi_2\delta_x)\otimes W$, which is 
free as an $\oL$-module, and 
$M\otimes_{\oL} L\cong \pi(\chi_1\delta_{x^{-1}},\chi_2\delta_x)\otimes W$. Moreover, since $\phi_x$ is the identity modulo $\pL^b$, we have 
$\Pi\star v \equiv \Pi\centerdot v \pmod{\varpi_L^b \mathcal D_1}$, for all $v\in \mathcal D_1$ and so 
the identity map $\mathcal D(x)_0 \rightarrow 
\mathcal D_0$ induces an isomorphism of diagrams $\mathcal D(x)\otimes_{\oL} \oL/\pL^b \cong \mathcal D\otimes_{\oL} \oL/\pL^b$. Now \eqref{Hmod} gives 
$H_0(\mathcal D)\otimes_{\oL}\oL/\pL^b\cong M\otimes_{\oL} \oL/\pL^b$. 
\end{proof}

Let $k\ge 2$ be an integer and $a_p\in \pL$, following Breuil \cite{breuil2} we define a filtered $\varphi$-module $D_{k,a_p}$: $D$ is a $2$-dimensional 
$L$-vector space with basis $\{e_1, e_2\}$, an $L$-linear automorphism $\varphi: D \rightarrow D$, given by 
$$\varphi(e_1)=p^{k-1} e_2, \quad \varphi(e_2)= -e_1+ a_p e_2;$$
a decreasing filtration $(\Fil^i D)_{i\in \ZZ}$ by $L$-subspaces, such that if $i\le 0$ then  $\Fil^i D=D$, if $1\le i\le k-1$ then $\Fil^i D=Le_1$, if $i\ge k$ 
then $\Fil^i D=0$. We set $V_{k,a_p}:= \Hom_{\varphi, \Fil^{\centerdot}}(D_{k,a_p}, B_{cris})$. Then $V_{k,a_p}$ is a $2$-dimensional $L$-linear absolutely irreducible 
crystalline representation of $\GG_{\Qp}:=\Gal(\Qpbar/\Qp)$ with Hodge-Tate weights $0$ and $k-1$. We denote by $\chi_{k,a_p}$ the trace character of $V_{k,a_p}$.
Since $\GG_{\Qp}$ is compact and the action is continuous, $\GG_{\Qp}$ stabilizes some $\oL$-lattice in $V_{k,a_p}$ and so $\chi_{k,a_p}$ takes values in $\oL$.

\begin{prop}\label{congruence} Let $m$ be the largest integer such that $m\le (k-2)/(p-1)$. Let $a_p, a'_p\in \pL$, and assume that 
$\val(a_p)> m$, $\val(a'_p)> m$. Let $n\ge em$ 
be an integer, where $e:=e(L/\Qp)$ is the ramification index. Suppose that $a_p\equiv a'_p \pmod{\pL^n}$, then 
$\chi_{k,a_p}(g)\equiv \chi_{k,a'_p}(g)\pmod{\pL^{n-em}}$ for all $g\in \GG_{\Qp}$.
\end{prop}
\begin{proof} This a consequence of a result of Berger-Li-Zhu \cite{bergerlizhu}. In \cite{bergerlizhu} they construct $\GG_{\Qp}$-invariant lattices $T_{k,a_p}$ in 
$V_{k,a_p}$. The assumption  $a_p\equiv a'_p \pmod{\pL^n}$ implies $T_{k,a_p}\otimes_{\oL} \oL/\pL^{n-em}\cong T_{k,a'_p}\otimes_{\oL}\oL/\pL^{n-em}$, 
see Remark 4.1.2 (2) in 
\cite{bergerlizhu}. This implies the congruences of characters.   
\end{proof}    

Let $k\ge 2$ be an integer and  $\lambda_1, \lambda_2\in L$, such that $\lambda_1+\lambda_2=a_p$ and $\lambda_1\lambda_2=p^{k-1}$ 
(enlarge $L$ if necessary). Assume that 
$\val(\lambda_1)\ge \val(\lambda_2)>0$. Let $\chi_1, \chi_2:\Qp^{\times}\rightarrow L^{\times}$ be unramified characters, with 
$\chi_1(p)=\lambda_1^{-1}$ and 
$\chi_2(p)=\lambda_2^{-1}$, let  $M$ be a finitely generated $\oL[G]$-module in $\pi(\chi_1, \chi_2|\centerdot|^{-1})\otimes W$, where 
$W:=\Sym^{k-2}L^2$. If $\lambda_1\neq \lambda_2$ 
then Berger-Breuil have shown that the unitary $L$-Banach space representation of $G$:
$$E_{k,a_p}:= L\otimes_{\oL} \underset{\longleftarrow}{\lim} \ M/\varpi_L^n M $$ 
is non-zero, topologically irreducible, admissible in the sense of \cite{scht}, and contains $\pi(\chi_1, \chi_2|\centerdot|^{-1})\otimes W$ as a dense $G$-invariant 
subspace, \cite[\S5.3]{bergerbreuil}. Moreover, the dual of $E_{k,a_p}$ is isomorphic to the representation of Borel subgroup $B$  constructed from the 
$(\varphi,\Gamma)$-module of $V_{k,a_p}$.

Let $\Rep_{\oL} G$ be the category of finite length $\oL[G]$-modules  with a central character, such that the action of $G$ is smooth (i.e. the stabilizer of 
a vector is an open subgroup of $G$.) Let $\Rep_{\oL}{\GG_{\Qp}}$ be the category of continuous representations of $\GG_{\Qp}$ on $\oL$-modules of finite length. 
Colmez in \cite[IV.2.14]{colmez1} has defined an exact covariant functor $\mathbf V: \Rep_{\oL} G \rightarrow \Rep_{\oL}\GG_{\Qp}$. 
The constructions in \cite{bergerbreuil} and \cite{colmez1} are mutually inverse to one another. 
This means if we assume $\lambda_1\neq \lambda_2$ and let $M$ be as above, then 
\begin{equation}\label{vkap}
V_{k,a_p}\cong L\otimes_{\oL} \underset{\longleftarrow}{\lim}\ \mathbf V(M/\varpi_L^n M).
\end{equation}
The fact that  $M/\varpi_L^n M$ is an $\oL[G]$-module of finite length follows from \cite[Thm A]{berger}. 

\begin{thm}\label{main} Assume that $p>2$, and let $\lambda=\pm p^{(k-1)/2}$, and $\chi:\Qp^{\times}\rightarrow L^{\times}$ a smooth character, with 
$\chi(p)=\lambda^{-1}$. Assume that there exists a $G$-invariant norm $\|\centerdot\|$ on 
$\pi(\chi, \chi|\centerdot|^{-1})\otimes W$, where $W:=\Sym^{k-2} L^2$. Let $E$ be the completion of $\pi(\chi, \chi|\centerdot|^{-1})\otimes W$ with respect to $\|\centerdot\|$. Then 
$E$ is non-zero, topologically irreducible, admissible Banach space representation of $G$. Moreover, if we let $E^0$ be the unit ball in $E$ then 
$$ V_{k,2\lambda}\otimes (\chi|\chi|)\cong L\otimes_{\oL} \underset{\longleftarrow} {\lim}\ \mathbf V( E^0/\varpi^n_L E^0).$$
\end{thm}
\begin{proof} Since the character $\chi|\chi|$ is integral, by twisting we may assume that $\chi$ is unramified. We denote the diagram 
$$\pi(\chi, \chi|\centerdot|^{-1})^{I_1}\otimes W\hookrightarrow \pi(\chi, \chi|\centerdot|^{-1})^{K_1}\otimes W$$
 by $D=(D_1\hookrightarrow D_0)$. Let $\mathcal D=(\mathcal D_1\hookrightarrow \mathcal D_0)$ be the diagram of $\oL$-modules 
with $\mathcal D_1=D_1\cap E^0$ 
and $\mathcal D_0=D_0\cap E^0$. Let $a\ge 1$ be the integer given by Proposition \ref{adjust}, for each $j\ge 0$, we fix $x_j\in 1+\pL^{a+j}$, $x_j\neq 1$ and a finitely 
generated $\oL[G]$-submodule $M_j$ in $\pi(\chi\delta_{x_j^{-1}}, \chi\delta_{x_j} |\centerdot|^{-1})\otimes W$, (which is then a free $\oL$-module), such that 
$$H_0(\mathcal D)\otimes_{\oL} \oL/\pL^{a+j}\cong M_j\otimes_{\oL} \oL/\pL^{a+j}.$$   
This is possible by Proposition \ref{adjust}. To ease the notation we set $M:=H_0(\mathcal D)$. Let $a_p(j):=\lambda x_j^{-1}+\lambda x_j$, $a_p:=2\lambda$ and 
let $m$ be the largest integer, such that $m\le (k-2)/(p-1)$. 
Since $p>2$, $x_j+x^{-1}_j$ is a unit in $\oL$, and so $\val(a_p(j))=\val(a_p)=(k-1)/2> m$. (Here we really need $p>2$.) 
Moreover, we have $a_p\equiv a_p(j) 
\pmod{\pL^{j+a+ em}}$, where $e:=e(L/\Qp)$ is the ramification index. Now since $x_j\neq 1$ we get that $\lambda x_j\neq \lambda x_j^{-1}$, and hence we may apply the 
results of Berger-Breuil to $\pi(\chi\delta_{x_j^{-1}}, \chi\delta_{x_j} |\centerdot|^{-1})\otimes W$. Let 
$T_{k,a_p(j)}:=\underset{\longleftarrow}{\lim}\ \mathbf V(M_j/\varpi_L^n M_j)$. Then 
\eqref{vkap} gives that $T_{k,a_p(j)}$ is a $\GG_{\Qp}$-invariant lattice in $V_{k,a_p(j)}$. Since $M\otimes_{\oL} \oL/\pL^{a+j}\cong M_j\otimes_{\oL} \oL/\pL^{a+j}$ 
we get 
\begin{equation}\label{congT}
\mathbf V(M/\varpi_L^{a+j} M)\cong \mathbf V(M_j/\varpi_L^{a+j} M_j)\cong T_{k,a_p(j)}\otimes_{\oL} \oL/\pL^{a+j}.
\end{equation}
Set $V:=L\otimes_{\oL} \underset{\longleftarrow}{\lim} \ \mathbf V(M/\varpi_L^n M)$. 
Then \eqref{congT} implies that $V$ is a $2$-dimensional $L$-vector space. Let $\chi_V$
be the trace character of $V$, then it follows from \eqref{congT} that 
$\chi_V \equiv \chi_{k,a_p(j)} \pmod{\pL^{a+j}}$. Since $a_p\equiv a_p(j)\pmod{\pL^{a+j+em}}$, 
Proposition \ref{congruence} says that 
$\chi_{k,a_p}\equiv \chi_{k,a_p(j)} \pmod{\pL^{a+j}}$. We obtain $\chi_V \equiv \chi_{k,a_p} \pmod{\pL^{a+j}}$, for all $j\ge 0$.
This gives us $\chi_V=\chi_{k,a_p}$. Since $V_{k,a_p}$ is irreducible, the equality of characters implies $V\cong V_{k,a_p}$.

Set $\widehat{M}:=\underset{\longleftarrow}{\lim} \ M/\varpi_L^n M$, and $E':=\widehat{M}\otimes_{\oL} L$. Since $M$ is a free $\oL$-module, 
we get an injection 
$M\hookrightarrow \widehat{M}$. In particular $E'$ contains $\pi(\chi, \chi|\centerdot|^{-1})\otimes W$ as a dense $G$-invariant subspace.
We claim that $E'$ is a topologically irreducible and admissible $G$-representation. Now \cite[Thm.4.1.1, Prop.4.1.4]{bergerlizhu} 
say that the semi-simplification of  
$T_{k, a_p(j)}\otimes_{\oL} k_L$ is  irreducible if $p+1 \nmid k-1$  and  isomorphic to 
$\begin{pmatrix} \mu_{\sqrt{-1}} & 0 \\ 0 & \mu_{-\sqrt{-1}}\end{pmatrix} \otimes \omega^{(k-1)/(p+1)}$, if $p+1 | k-1$, where 
$\mu_{\pm\sqrt{-1}}$ is the unramified character 
sending arithmetic Frobenius to $\pm \sqrt{-1}$, and $\omega$ is the cyclotomic character. Then \cite[Thm A]{berger} 
implies that if  
$p+1\nmid k-1$ then 
$M_j \otimes_{\oL} k_L$ is an irreducible supersingular representation of $G$, and 
if $p+1 | k-1$ then the semi-simplification of $M_j\otimes_{\oL} k_L$ is a direct sum 
of two irreducible principal series. The irreducibility of principal series follows from \cite[Thm. 33]{bl}, since 
$\sqrt{-1}\neq \pm 1$, as $p>2$. Since 
$M\otimes_{\oL} k_L\cong M_j\otimes_{\oL} k_L$, we get that $M\otimes_{\oL} k_L$ is an admissible  representation of $G$ 
(so that for every open subgroup $\mathcal U$ 
of $G$, the space of $\mathcal U$-invariants is finite dimensional). This implies that  $E'$ is admissible.     

Suppose that $E_1$ is a closed $G$-invariant subspace of $E'$ with $E'\neq E_1$. Let $E_1^0:= E_1\cap \widehat{M}$. 
We obtain a $G$-equivariant  injection 
$E_1^0\otimes_{\oL} k_L\hookrightarrow M\otimes_{\oL} k_L$. 
If $E_1^0\otimes_{\oL} k_L=0$ or $M\otimes_{\oL} k_L$ then Nakayama's lemma gives $E_1^0=0$ and 
$E_1^0=\widehat{M}$, respectively. If $p+1\nmid k-1$ then $M\otimes_{\oL} k_L$ is irreducible and we are done. 
If $p+1 | k-1$ then $E_1^0\otimes_{\oL} k_L$ is an 
irreducible principal series, and so $\mathbf V(E_1^0\otimes_{\oL} k_L)$ is one dimensional,
\cite[IV.4.17]{colmez1}. 
But then  
$V_1:=L\otimes_{\oL} \underset{\longleftarrow}{\lim} \ \mathbf V(E_1^0/\varpi_L^n E_1^0)$
is a $1$-dimensional subspace of $V_{k, a_p}$ stable under the action of $\GG_{\Qp}$. Since $V_{k,a_p}$ is irreducible we obtain a contradiction.

Since $E'$ is a completion of $\pi(\chi, \chi|\centerdot|^{-1})\otimes W$ with respect to a finitely generated $\oL[G]$-submodule, it is the universal completion, see 
eg \cite[Prop. 1.17]{emerton}. 
In particular, we obtain a non-zero $G$-equivariant map of $L$-Banach space representations $E'\rightarrow E$, but since $E'$ is irreducible and  
$\pi(\chi, \chi|\centerdot|^{-1})\otimes W$ is dense in $E$, this map is an isomorphism. 

\end{proof}

\begin{cor}\label{univfg} Assume that $p>2$, and let $\chi:\Qp^{\times}\rightarrow L^{\times}$ a smooth character with 
$\chi(p)^2p^{k-1}=1$. Assume that there exists a $G$-invariant norm $\|\centerdot\|$ on 
$\pi(\chi, \chi|\centerdot|^{-1})\otimes W$, where $W:=\Sym^{k-2} L^2$. Then every bounded $G$-invariant $\oL$-lattice 
in $\pi(\chi, \chi|\centerdot|^{-1})\otimes W$ is finitely generated as an $\oL[G]$-module.
\end{cor}
\begin{proof} The existence of a $G$-invariant norm implies that the universal completion is non-zero. It follows from
Theorem \ref{main} that the universal completion is topologically irreducible and admissible. The assertion follows from the proof of 
\cite[Cor. 5.3.4]{bergerbreuil}.
\end{proof} 

For the purposes of \cite{prep} we record the following corollary to the proof of Theorem \ref{main}.

\begin{cor} Assume $p>2$, and let $\chi:\Qp^{\times} \rightarrow L^{\times}$ be a smooth character, such that 
$\chi^2(p)p^{k-1}$ is a unit in $\oL$. 
Assume there exists a unitary $L$-Banach space representation $(E,\|\centerdot\|)$ of $G$ containing 
$(\Indu{B}{G}{\chi\otimes \chi|\centerdot|^{-1}})\otimes \Sym^{k-2} L^2$ as a dense $G$-invariant subspace, such that 
$\|E\|\subseteq |L|$. Then there exists $x\in 1+\pL$, $x^2\neq 1$ and a unitary completion 
$E_x$ of $(\Indu{B}{G}{\chi\delta_x\otimes \chi\delta_{x^{-1}}|\centerdot|^{-1}})\otimes \Sym^{k-2} L^2$, 
such that $E^0\otimes_{\oL} k_L \cong E^0_x \otimes_{\oL} k_L$, 
where $E^0_x$ is the unit ball in $E_x$ and $E^0$ is the unit ball in $E$.
\end{cor} 
\begin{proof} Let $\pi:=\Indu{B}{G}{\chi\otimes \chi|\centerdot|^{-1}}$ and 
$M:=(\pi\otimes W)\cap E^0$. Now 
$M\cap \varpi_L E^0=(\pi\otimes W)\cap \varpi_L E^0= \varpi_L M.$
So we have a $G$-equivariant injection $\iota:M/\varpi_L M\hookrightarrow E^0/\varpi_L E^0$. We claim that $\iota$ 
is a surjection. Let $v\in E^0$, since $\pi\otimes W$ is dense in $E$, there exists a sequence $\{v_n\}_{n\ge 1}$ 
in  $\pi\otimes W$ such that $\lim v_n =v$. We also have $\lim \|v_n\|=\|v\|$. Since $\|E\|\subseteq |L|\cong \mathbb Z$, there
exists $m\ge 0$ such that $v_n\in M$, for all $n\ge m$. This implies surjectivity of $\iota$. So we get $M\otimes_{\oL} k_L\cong 
E^0\otimes_{\oL} k_L$. 

By Corollary \ref{univfg} we may find $u_1,\ldots, u_n \in M$ which generate
$M$ as an $\oL[G]$-module. Further, $u_i=\sum_{j=1}^{m_i} v_{ij}\otimes w_{ij}$ with $v_{ij}\in \pi$ and $w_{ij}\in W$. Since
$\pi$ is a smooth representation of $G$ there exists an integer $c\ge 1$ such that $v_{ij}$ is fixed by $K_c$ for all 
$1\le i\le n$, $1\le j\le m_i$. Set 
$$\DD:=((\pi^{I_c}\otimes W)\cap M \hookrightarrow (\pi^{K_c}\otimes W)\cap M), 
\quad D:= (\pi^{I_c}\otimes W \hookrightarrow \pi^{K_c}\otimes W)$$
 and let $M'$ be the image of $H_0(\DD)\hookrightarrow H_0(D)\cong \pi\otimes W$. It follows from \eqref{homology} that $M'$ is generated by 
$(\pi^{K_c}\otimes W)\cap M$ as an $\oL[G]$-module. Hence, $M'\subseteq M$. By construction $(\pi^{K_c}\otimes W)\cap M$
contains $u_1, \ldots u_n$, and so $M\subseteq M'$. In particular, $H_0(\DD)\otimes_{\oL} k_L \cong M\otimes_{\oL} k_L$. The assertion 
follows from the proof of Theorem \ref{main}.
\end{proof}  

\section{Existence}\label{exist}
Recent results of Colmez, which appeared after the first version of this  note, imply the existence of a $G$-invariant norm on
$(\Indu{B}{G}{\chi\otimes \chi|\centerdot|^{-1}})\otimes \Sym^{k-2}L^2$, $\chi^2(p)p^{k-1}\in \oL^{\times}$, thus making 
our results unconditional. We briefly explain this. 

We continue to assume $p>2$, $k\ge 2$ an integer and $a_p=2p^{(k-1)/2}$. The representation $V_{k, a_p}$ of $\mathcal G_{\Qp}$ 
sits in the $p$-adic family 
of Berger-Li-Zhu, \cite[3.2.5]{bergerlizhu}. Moreover, all the other points in the family 
correspond to the crystalline representations with distinct Frobenius eigenvalues, to which the theory of \cite{bergerbreuil}
applies.  Hence \cite[II.3.1, IV.4.11]{colmez1} implies that there exists an irreducible 
unitary $L$-Banach space representation $\Pi$ of $\GL_2(\Qp)$, such that $\mathbf V (\Pi)\cong V_{k, a_p}$. If $p\ge 5$ or $p=3$ and 
$k\not\equiv 3\pmod{8}$ and $k\not\equiv 7 \pmod{8}$, the existence of such $\Pi$ also follows from
\cite{kisin}. It follows from \cite[VI.6.46]{colmez1} that the set of locally algebraic vectors $\Pi^{alg}$ of $\Pi$ 
is isomorphic to $(\Indu{B}{G}{\chi\otimes \chi|\centerdot|^{-1}})\otimes \Sym^{k-2}L^2$, where $\chi: \Qp^{\times}\rightarrow L^{\times}$
is an unramified character with $\chi(p)=p^{-(k-1)/2}$. The  restriction of  the $G$-invariant norm of $\Pi$ to $\Pi^{alg}$ solves the problem.
Moreover, if $\delta:\Qp^{\times}\rightarrow L^{\times}$ is a unitary character then we also obtain a $G$-invariant norm on 
$\Pi^{alg}\otimes\delta\circ \det$.

\end{document}